\documentclass[11pt]{amsart}
\usepackage{graphicx}
\usepackage[usenames]{color}
\usepackage{amssymb,amscd}

\usepackage{ifthen}

\newcommand{\showcomments}{yes}

\newsavebox{\commentbox}
{\ifthenelse{\equal{\showcomments}{yes}}%
{\footnotemark
        \begin{lrbox}{\commentbox}
        \begin{minipage}[t]{1.25in}\raggedright\sffamily\tiny
        \footnotemark[\arabic{footnote}]}
{\begin{lrbox}{\commentbox}}}%
{\ifthenelse{\equal{\showcomments}{yes}}%
{\end{minipage}\end{lrbox}\marginpar{\usebox{\commentbox}}}
{\end{lrbox}}}
\newcommand{\bea}{\begin{eqnarray*}}
\newcommand{\eea}{\end{eqnarray*}}

\newcommand{\bm}{\begin{pmatrix}}
\newcommand{\fm}{\end{pmatrix}}
\newcommand{\ra}{\rightarrow}

\newcommand{\iso}{\cong}

\DeclareMathOperator{\Aut}{Aut}
\DeclareMathOperator{\nil}{nil}

\DeclareMathOperator\tr{tr}
\DeclareMathOperator\Isom{Isom}

\newcommand\Z{\mathbb Z}

\newcommand\Q{\mathbb Q}
\newcommand\R{\mathbb R}

\newcommand{\g}{\gamma}

\newtheorem{theorem}{Theorem}[section]
\newtheorem{problem}{Problem}[section]
\newtheorem{claim}{Claim}[section]

\newtheorem{prop}[theorem]{Proposition}

\newtheorem{lemma}[theorem]{Lemma}

\title{\footnote{} Day's fixed point theorem, Group Cohomology\\ and   Quasi-isometric rigidity}
\author{Tullia Dymarz}
\address{Department of Mathematics,
University of Wisconsin-Madison, 480 Lincoln Drive,  Madison, WI 53706} \email{dymarz@math.wisc.edu}

\author{Xiangdong Xie}
\address{Department of Mathematics and Statistics,
Bowling Green State University,
Bowling Green, OH 43403} \email{xiex@bgsu.edu}
\begin{document}
\maketitle

\begin{abstract}

In this note we explain how Day's fixed point theorem can be used to conjugate certain groups of   biLipschitz maps  of a metric space  into 
special subgroups like similarity groups.
  In particular,  we  use Day's theorem to  establish   Tukia-type theorems  and 
 to give  new proofs of quasi-isometric  rigidity  results.

\end{abstract}

\section{Introduction}\label{intro}

In \cite{T} Tukia proves that any uniform group of quasiconformal maps  of $S^n$ for $n \geq 3$ that acts cocompactly on the space of distinct triples of $S^n$ is conjugate into the conformal group. For $n=2$ this was already proved by Sullivan in \cite{S} without the added assumption on triples.
Since the boundary of the real hyperbolic space $\mathbb{H}^n$ can be identified with $S^{n-1}$ and since quasi-isometries (resp. isometries) of $\mathbb{H}^n$ induce quasiconformal (resp. conformal)  boundary maps,  Tukia's   theorem  is used as a key ingredient in the proof of quasi-isometric rigidity of uniform lattices in the isometry groups of real hyperbolic spaces. (For more details see \cite{CC}).



The main goal of this paper is to extend Tukia's result to maps on boundaries 
of different $\delta$-hyperbolic spaces; namely to boundaries of certain 
 \emph{negatively curved homogeneous spaces} and certain \emph{millefeuille spaces}.  We leave the description of the geometry of these negatively curved homogeneous spaces (resp.  millefeuille spaces) to Section \ref{app:sec}.
For the statement of our theorems we only need to know that their visual boundaries can be identified with $N \cup \{\infty\}$  (resp. $(N \times \Q_m) \cup \{\infty\}$) where $N$ is a nilpotent Lie group (and $\Q_m$ is the $m$-adics). Additionally our maps fix the point at infinity  $\infty$ and so we consider \emph{parabolic visual metrics} on $N$ (resp. $N \times \Q_m$) instead of the standard visual metrics on the whole boundary.

The first case we treat is when $N=F^n$  (which has  dimension $n+1$)  is a model Filiform group and the parabolic visual metric is the usual  Carnot-Carath\'eodory metric. 
Recall that  a group $\Gamma$ of quasiconformal maps of a metric space $X$ is called a {\it{uniform 
quasiconformal group}} if there is some $K\ge 1$ such that every element $\gamma$ of 
$\Gamma$ is  $K$-quasiconformal.
The following theorem is proved in Section \ref{filiform}. 

\begin{theorem}\label{fili}
Let $\Gamma\subset QC(F^n)$  ($n\ge 3$)  be a locally compact uniform quasiconformal group  of the 
  $n$-th  model Filiform  group. Then there is some $f\in QC(F^n)$ such that $f\Gamma f^{-1}$ is a  conformal  group.
\end{theorem}

The second case we consider is when $N=\R^n$  but where the parabolic visual metric, denoted $d_A$, is not the usual metric on $\R^n$ but is instead determined 
   by a  matrix $A$ that is  diagonalizable over complex numbers and whose  eigenvalues  have positive real parts.
     This metric is defined in Section \ref{app:sec}. In this case, the analogue of quasiconformal maps are {quasisimilarity} maps. 
  A bijection $F: (X,d)\ra (X,d)$ of a metric space is a   $(M, C)$-\emph{quasisimilarity} ($M\ge 1$, $C>0$)   if  
  $$\frac{C}{M} \cdot  d(x_1, x_2)\le d(F(x_1), F(x_2))\le MC \cdot  d(x_1, x_2)$$
    for all $x_1, x_2\in X$. 
When $M=1$ the map is  called  a \emph{similarity}. 
 A group $\Gamma$ of quasisimilarities of a metric space is {\it{uniform}} if there is some $M\ge 1$ such that  every element
  $\gamma$ of $\Gamma$ is a   $(M, C_\gamma)$-quasisimilarity ($C_\gamma$ may depend on $\gamma$).  
Note that a quasisimilarity map is actually just a biLipschitz map but a uniform group of quasisimilarities is not the same as a uniform group of biLipschitz maps. 

\begin{theorem}\label{ProperTukiaThm} 
  Let $A$ be as above and  $\Gamma$  a 
locally compact    uniform  quasisimilarity  group   
of $(\R^n , d_A)$.  
Suppose in addition that $\Gamma$ is amenable and acts cocompactly on the space  of distinct pairs of $\R^n$.  Let $\bar{A}$ denote  the matrix obtained from the  Jordan form of $A$
  by replacing  the eigenvalues with their real parts. 
Then there exists  a biLipschitz map $f: (\R^n , d_A) \to (\R^n, d_{\bar{A}})$  such 
    that $f\Gamma f^{-1}$ is a group of similarities of $(\R^n, d_{\bar{A}})$. 
\end{theorem}


{\bf Remarks.} 
\begin{enumerate}

\item Both Theorems \ref{fili} and \ref{ProperTukiaThm} can be thought of as \emph{Tukia-type} theorems and like Tukia's original theorem can be used to prove quasi-isometric rigidity results; this time for lattices in certain solvable Lie groups (see Section \ref{app:sec}). 
\item    Theorem \ref{ProperTukiaThm} extends partial results that can be found in \cite{D1}.

\item The two extra assumptions  in Theorem \ref{ProperTukiaThm}, amenability and cocompact action on pairs, are not an obstacle for the quasi-isometric rigidity results, nevertheless it would be interesting to know if they can be removed in certain cases. For Tukia's original theorem on quasiconformal maps of $S^n$ cocompactness on triples is necessary when $n>2$ (see for example \cite{T2}). We only need to assume cocompactness on distinct pairs since our groups fix the point at infinity.  For Theorem \ref{fili} amenability and cocompactness on distinct pairs are not needed as  assumptions because from \cite{X2} we know that maps in $QC(F^n)$ have a particularly nice form  which allows us to deduce that $QC(F^n)$ is itself  solvable  
 (see Claim \ref{claim}).
\item Theorem \ref{fili} could have also been stated as a theorem about uniform quasisimiliary groups since in \cite{X2} it is proved that all quasiconformal maps are in fact quasisimilarities.
\item Theorem \ref{ProperTukiaThm} can be extended to uniform quasisimilarity groups of $\R^n \times \Q_m$     
but we will leave the statement of this result to Theorem \ref{ProperTukia2} in Section \ref{app:sec}.  
\end{enumerate}

%
%
%


\subsection{Day's Theorem and bounded cohomology}

One key feature in studying quasiconformal (resp. quasisimilarity) maps of boundaries of negatively curved homogeneous spaces is that often differentiability fails in certain directions. This forces  quasiconformal and quasisimilarity maps to preserve certain families of foliations. While it is sometimes possible to prove that these maps can be conjugated to act by similarities along the leaves of these foliations, this does not automatically imply that the map is a global similarity map. To solve this problem we appeal to Day's theorem: 

{\bf Theorem. (Day's fixed point theorem) }\cite{Day}
\emph{ Let  $K$ be a compact convex subset  of a locally convex  topological   vector 
space  $E$   and let $\Gamma$ be a locally compact group that acts on $K$ by  affine transformations. If $\Gamma$ is amenable 
  and the action  $\Gamma\times K\ra K,  (\gamma, x)\mapsto \gamma\cdot x,$ is separately continuous,     
then the action of $\Gamma$ has a global fixed point. }

Roughly speaking, the relation between Day's fixed point theorem and a Tukia-type theorem is the following: on the one hand, finding a conjugating map corresponds to the vanishing of the first (bounded) group cohomology; on the other hand there is a well known connection between the vanishing of the first group cohomology and the existence of a global fixed point for an affine action.
This relation is explained in detail in Section \ref{fixedco}.

We now give a simple example to illustrate how Day's theorem can be used to prove a Tukia-type theorem. This example is the simplest case of Theorem \ref{conjtosim} from Section \ref{leaf} which is later used to prove Theorems \ref{fili} and \ref{ProperTukiaThm} in Sections \ref{filiform} and \ref{app:sec}. 

{\bf Example.} Suppose $\Gamma$ is a  locally compact uniform  quasisimilarity group   of $\R^2$    and 
    each $\g\in \Gamma$ has the form 
  $$\g(x,y)=(a_\g(x+c_\g+h_\g(y)), b_\g(y+d_\g)),$$
  where   $a_\g, b_\g\in \R\backslash\{0\}, c_\g,  d_\g\in \R$ and  $h_\g: \R  \ra \R$ is a Lipschitz map that satisfies $h_\g(0)=0$.     It is easy to check that 
   each $\g$ is a biLipschitz map of $\R^2$.  


%
%
In the above example, the locally convex topological vector space is 
  $$E=\{h: \R \ra \R  \;\;  \text{is Lipschitz and}\;\;   h(0)=0\}.$$
 Here $E$ is equipped with the topology of  pointwise convergence.    Similar to the Banach- Alaoglu theorem,  the closed bounded subsets are compact 
in the  topology of  pointwise convergence.    Here the boundedness  is with respect to the norm on $E$   given by the Lipschitz constant.  
  The affince action of $\Gamma$ (actually of the opposite group $\Gamma^*$ of $\Gamma$)  on $E$ is
given by 
$\phi(\g) h=\pi_\g h+h_\g$  for $\gamma\in \Gamma$ and $h\in E$,    where 
$$\pi_\g h(y)=a_\g^{-1}  h (b_\g(y+d_\g))-a_\g^{-1}  h (b_\g d_\g).$$
  The uniformity condition on $\Gamma$  implies  that this action has bounded orbits.   The compact convex subset $K$ is the closed 
 convex hull of an orbit.   

Of course  the claim for the above particular example also   follows from Sullivan's theorem \cite{S}.  
    But the approach in this paper applies in all dimensions and in situations where differentiability of Lipschitz maps fails such as in  Theorem  
     \ref{ProperTukiaThm} above.

%

\subsection{Applications.} 
As mentioned above, Theorems \ref{fili} and \ref{ProperTukiaThm} can be used to prove   
quasi-isometric rigidity results for 
   solvable Lie groups  and their lattices. 
 Additionally, they can be used to prove results on \emph{envelopes} of certain finitely generated solvable groups.  These results are described in Section \ref{app:sec}. 
 As an example, here we mention a quasi-isometric rigidity result for solvable Lie groups.

\noindent
{\bf{Theorem \ref{s5.3.2}}}
{ \it{Let $F^n$  ($n\ge 3$) be the  $n$-th  model Filiform group and $S=F^n\rtimes \R$ the semi-direct product, where the action of $\R$ on $F^n$ is by standard Carnot dilation.
  Then any connected and simply connected solvable Lie group 
 $G$ quasi-isometric to $S$ is isomorphic to $S$. }}

%
%

{\bf{Acknowledgement}}.   This paper was   initiated  when the authors were  attending the quarter program
 ``Random walks on groups"  at the Institute of Henri Poincare from January  to March 2014.   They would like to thank IHP for its hospitality and wonderful working conditions.    The first author acknowledges support from National Science Foundation grant DMS-1207296. The second author acknowledges support from National Science Foundation grant DMS--1265735  and   Simons Foundation grant \#315130. The authors would also like to thank Dave Witte Morris for providing Proposition \ref{witte:prop} and David Fisher for useful conversations.  They thank Yves de Cornulier for comments on an earlier version.   The first author would also like to thank Pierre Pansu for useful conversations.

\section{Vanishing of   Cohomology  and  Fixed Points}\label{fixedco}

In this section we recall the connection  between   vanishing of 
  group cohomology  in dimension one and   the  existence of  global   fixed points  of affine actions.
  We include it here to make this note more self contained.

Let $\Gamma$ be a group and  
 $E$ a $\Gamma$-module. 
  So $E$   is an abelian group and   a homomorphism $\pi: \Gamma\ra \text{Aut}(E)$ is fixed,  where  $\text{Aut}(E)$   is the group of automorphisms of $E$.
      We denote the image of $\g\in \Gamma$ under $\pi$ by $\pi_\g$. 
 A  map $b:  \Gamma \ra  E$,  $\g\mapsto  b_\g$,   is  called  a   $1$-{\it{cocycle}}   if 
$$b_{\g_1\g_2}=b_{\g_1}+\pi_{\g_1} b_{\g_2}\;\;\text{for all}\;\;  \g_1, \g_2\in \Gamma.$$
  Here we denote the group operation in $E$ by addition. 
We  only consider $1$-cocycles, so we will simply call $b$ a cocycle.
  The set  $Z^1(\Gamma, E)$  of all  cocycles is an   abelian group under addition
 given  by $(b+\tilde b)_\g=b_\g+\tilde b_\g$.
A  map $b:  \Gamma \ra  E$ is  called  a  {\it{coboundary}}   if there exists 
$v_0\in E$ such that 
 $$b_\g=v_0-\pi_\g v_0\;\;\text{for  any}\;\;   \g\in \Gamma.$$
  It is easy to check that a coboundary is a cocycle
  and the  set $B^1(\Gamma, E)$ of all coboundaries is a subgroup of $Z^1(\Gamma, E)$.
   The first cohomology of $\Gamma$ with coefficients in $E$ is 
 $$H^1(\Gamma, E)=Z^1(\Gamma, E)/{B^1(\Gamma, E)}.$$

When $E$ is also a metric space (for example, a normed vector space)
  and $\Gamma$ acts on $E$ by uniformly biLipschitz maps (that is, there exists   some $M\ge 1$ such that $d(v_1, v_2)/M\le d(\pi_\g v_1,   \pi_\g v_2)\le M \cdot d(v_1, v_2)$ for all $\g\in \Gamma$, $v_1,  v_2\in E$),
 a   cocycle is said to be  bounded if its image is a bounded subset of $E$. 
  The set $Z^1_b(\Gamma, E)$ of all bounded cocycles is an abelian group under addition. 
 Notice that every    coboundary is a  bounded  cocycle so that $B^1_b(\Gamma, E)= B^1(\Gamma,E)$. 
 The first bounded cohomology of $\Gamma$ with coefficients in $E$ is
  $$H^1_b(\Gamma, E)=Z^1_b(\Gamma, E)/{B^1(\Gamma, E)}.$$

  A  map   $f: E\ra E$ is called an \emph{affine map} if there exists some automorphism 
 $A: E\ra E$ and some  $v_0\in E$ such that $f(v)=A v+v_0$ for all $v\in E$. 
 Let $\text{Aff}(E)$ be the group of all affine maps of  $E$. 
 An \emph{affine action} of $\Gamma$ on $E$ is 
a  homomorphism $\phi: \Gamma\ra  \text{Aff}(E)$.  
 Given an affine action 
$\phi: \Gamma\ra  \text{Aff}(E)$, we define a map
 $\pi: \Gamma\ra \text{Aut}(E)$ by
 $\pi_\g=A$, where $A\in \text{Aut}(E)$ is determined by 
 $\phi(\g)(v)=Av +v_0$. 
  It is easy to check that $\pi$ is a homomorphism.  
   We call $\pi$ the \emph{linear part}  
    of the affine action $\phi$. 


\begin{lemma}\label{l1.1}
Let $E$ be a $\Gamma$-module  with the  $\Gamma$ action on $E$ 
  given by $\pi: \Gamma\ra \text{Aut}(E)$.  \newline
 (1) There is a  bijection between   $Z^1(\Gamma, E)$  and the set of affine actions of $\Gamma$ on $E$ with  linear part  $\pi$;\newline
(2) A   cocycle is a coboundary if and only if the corresponding affine action 
 has a   global fixed point;\newline
(3)   $H^1(\Gamma, E)=0$ if and only if  every affine action with linear part $\pi$ has a   global   fixed point.

\end{lemma}

\begin{proof}
(1)  Given a cocycle  $b: \Gamma \ra E$,  define an affine action  
  $\phi: \Gamma\ra  \text{Aff}(E)$
 of $\Gamma$ on $E$ by
 $$\phi(\g)(v)=\pi_\g  v+b_\g    \;\;  {\text{for}}   \;\;   \g\in \Gamma,  v\in E.$$
Indeed,  $\phi(\g)$ is an affine map and for $\g_1, \g_2\in \Gamma$, we have
 \begin{align*} \phi(\g_1)\circ \phi(\g_2)(v)&=\phi(\g_1)(\pi_{\g_2} v+b_{\g_2})\\
&=\pi_{\g_1} (\pi_{\g_2}v   +b_{\g_2})+b_{\g_1}\\
& =\pi_{\g_1}\pi_{\g_2} v+(b_{\g_1}+\pi_{\g_1} b_{\g_2})\\
&=\pi_{{\g_1}{\g_2}}  v+b_{\g_1\g_2}\\
& =\phi(\g_1\g_2)(v).
\end{align*}

Conversely,  assume   $\phi: \Gamma\ra  \text{Aff}(E)$  is an affine action of  
 $\Gamma$ on $E$   with linear part  $\pi$. 
 For any $\g\in \Gamma$,  there is some  $b_\g\in E$ such that  for any $v\in E$,
 $$\phi(\g)(v)=\pi_\g v+b_\g.$$
  For any $\g_1, \g_2\in \Gamma$, the condition $\phi(\g_1\g_2)=\phi(\g_1)\circ \phi(\g_2)$ implies:
  $b_{\g_1\g_2}=b_{\g_1}+\pi_{\g_1}  b_{\g_2}$. 
    Hence $b: \Gamma   \ra E$ is a 
   cocycle.     Now it is easy to see  that the above correspondence  $b\ra \phi$ is a  bijection 
between   $Z^1(\Gamma, E)$  and the set of affine actions of $\Gamma$ on $E$ with  linear part  $\pi$.  

(2) First assume  that   
  $b: \Gamma\ra E$ is a coboundary.  Then there is some $v_0\in E$ such that
 $b_\g=v_0-\pi_\g v_0$  for all $\g\in \Gamma$.    Then for any $\g\in \Gamma$, we have  
$\phi(\g)(v_0)=\pi_\g v_0+b_\g=v_0$. Hence $v_0$ is a global fixed point of the affine action of $\Gamma$ corresponding to the cocycle $b$.  Conversely,
 assume the affine  action  has a global fixed point $v_0$. Then  $v_0=\phi(\g)(v_0)=\pi_\g   v_0+b_\g$ for any $\g\in \Gamma$.   Hence $b_\g=v_0-\pi_\g v_0$   for all $\g\in \Gamma$  
  and $b$ is a coboundary.

(3)  follows from  (1)  and (2). 

\end{proof}

Similarly we have

\begin{lemma}\label{l1.2}
Let $E$ be a $\Gamma$-module  with the  $\Gamma$   action on $E$  given by $\pi: \Gamma\ra \text{Aut}(E)$. 
   Assume $E$ is also a metric space and   $\Gamma$ acts on $E$ by uniformly biLipschitz maps. \newline
 (1) There is a  bijection between   $Z_b^1(\Gamma, E)$  and the set of affine actions of $\Gamma$ on $E$ with   linear part   $\pi$  and  bounded orbits;\newline
(2) A   bounded   cocycle is a coboundary if and only if the corresponding affine action 
 has a  global  fixed point;\newline
(3)   $H_b^1(\Gamma, E)=0$ if and only if  every affine action with    linear part   $\pi$  and  bounded orbits    has a global  fixed point.

\end{lemma}

\section{BiLipschitz maps of $\R^n \times Y$}\label{leaf}

In this section we explain how Day's fixed point  theorem can be used to conjugate certain groups of    biLipschitz maps of metric spaces of the form $\R^n \times Y$ into   groups of  \emph{similarities}.

Recall that a similarity  is a map $f:(X,d) \to (X,d)$  such that there is a constant
  $c_f>0 $  for which 
 $d(f(x),f(x'))=c_f d(x,x')$   for all $x,x' \in X$. 


%
 Let $(Y,d)$ be a metric space.  Let $0<\beta\le 1$ and $\R^n$ be  equipped with the metric 
 $|p-q|^\beta$ ($p, q\in \R^n$), where $|\cdot |$ is the usual Euclidean norm. 
 Let  $y_0\in Y$  be a  fixed    base point. 
Let 
 %
 %
$$E=\{ h: (Y, d)   \ra (\R^n, |\cdot|^\beta) \;\;   \text{is Lipschitz and}\;\; 
       h(y_0)=0\}.$$ 
   It is not hard to check that $E$ is a Banach space with  the following norm:
 $$||h||=\sup_{y_1\not= y_2\in Y}\frac{|h(y_1)-h(y_2)|}{d(y_1, y_2)^{\frac{1}{\beta}}}.$$


%


  Let $\R^n\times Y$ be equipped with the metric 
 $$d((x_1, y_1), (x_2, y_2))=\max\{|x_1-x_2|^\beta, d(y_1, y_2)\}\;\; \text{for}\;\; (x_1, y_1), (x_2, y_2)\in \R^n\times Y.$$
Suppose that a  group  $\Gamma$ acts on $\R^n \times Y$ by  biLipschitz maps,
  and for every $\g\in \Gamma$,   there exist 
 $a_\g \in \R_{+}, A_\g \in \text{O}(n)$, $x_\g\in \R^n$, $h_\g\in E$   and a similarity   $\sigma_\g: Y\ra Y$  of $Y$  such that  
$$ \g(x,y) =( a_\g A_\g(x +x_\g+ h_\g(y)), \sigma_\g(y))\;\;\text{for}\;\; (x,y)\in \R^n\times Y.$$
%

\begin{lemma}\label{hislip}
Suppose  $\Gamma$ acts on  $\R^n \times Y$  as described above,  and acts as a 
uniform quasisimilarity group.
Then:\newline
 (1)  for any $\g\in \Gamma$, the similarity constant of $\sigma_\g$ is $a_\g^\beta$; that is, 
  $d(\sigma_\g(y_1), \sigma_\g(y_2))= a_\g^\beta\cdot d(y_1, y_2)$ for any $y_1, y_2\in Y$;\newline
  (2)    there is a constant $C>0$ such that   $h_\g: (Y, d)   \ra (\R^n, |\cdot|^\beta)$   
is $C$-Lipschitz for every $\g \in \Gamma$.

\end{lemma}

\begin{proof}
  (1)   Since $\Gamma$ acts  as a uniform quasisimilarity group, there is a constant $M\ge 1$ such that 
every $\gamma \in \Gamma$ is   a $(M,C_\gamma)$-quasisimilarity.  
   Denote by $\tilde a_\g$ the similarity constant of 
 $\sigma_\g$.     We shall show $\tilde a_\g=a_\g^\beta$.   
   Fix $x\in \R^n$  and  $y_1\not= y_2\in Y$,    and    let $p=(x, y_1)$, $q=(x, y_2)$.
  Then  $d(p, q)=d(y_1, y_2)$.  
     Hence:   
$$\tilde a_\g\cdot d(y_1, y_2) =d(\sigma_\g(y_1), \sigma_\g(y_2))\le d(\g(p), \g(q))\le  M C_\g\cdot d(p, q)=M C_\g \cdot d(y_1, y_2).$$
 So    $\tilde a_\g\le M C_\g$.  By considering  $\g^{-1}$ we obtain 
  $$\frac{1}{\tilde a_\g}=\tilde a_{\g^{-1}}\le M C_{\g^{-1}}=M\cdot \frac{1}{C_\g}.$$
    Hence  
\begin{equation}\label{e3.1}
 \frac{C_\g}{M}\le  \tilde a_\g\le M C_\g \;\;\text{for all}\;\; \g\in \Gamma.
\end{equation}
    Next we    fix   $y\in Y$ and $x_1\not=x_2\in \R^n$,   and  let $p=(x_1, y)$,  $q=(x_2, y)$.   Then  $d(p, q)=|x_1-x_2|^\beta$  and  $d(\g(p), \g(q))=(a_\g\cdot |x_1-x_2|)^\beta$.  
The quasisimilarity condition on $\g$ implies:
$$ \frac{C_\g}{M}\cdot  |x_1-x_2|^\beta= \frac{C_\g}{M}\cdot  d(p, q)\le  (a_\g\cdot |x_1-x_2|)^\beta=d(\g(p), \g(q))\le M C_\g \cdot d(p, q)=M C_\g \cdot |x_1-x_2|^\beta.$$
  Hence  
\begin{equation}\label{e3.2}
 \frac{C_\g}{M}\le  a_\g^\beta\le M C_\g \;\;\text{for all}\;\; \g\in \Gamma.
\end{equation}
  Inequalities  (\ref{e3.1}) and  (\ref{e3.2})   imply 
\begin{equation}\label{e3.3}
 \frac{1}{M^2}\le  \frac{\tilde a_\g}{a_\g^\beta}  \le M^2  \;\;\text{for all}\;\; \g\in \Gamma.
\end{equation}
 In particular,  (\ref{e3.3})  holds for $\g^m$ for any $m\in \mathbb Z$. Since 
 $a_{\g^m}=a_\g^m$ and  $\tilde a_{\g^m}=\tilde a_\g^m$, 
    we must have $\tilde a_\g=a_\g^\beta$.

(2)         Let   $y_1,   y_2\in Y$   be arbitrary.     Fix    $x\in \R^n$   and   let $p=(x, y_1)$, $q=(x, y_2)$.
  Then  $d(p, q)=d(y_1, y_2)$.  
  The quasisimilarity condition on $\g$ implies:
   \begin{align*}
(a_\g\cdot |h_\g(y_1)-h_\g(y_2)|)^\beta&=|a_\g A_\g (x+x_\g+h_\g(y_1))-a_\g A_\g (x+x_\g+h_\g(y_2))|^\beta\\
 & \le
d(\g(p), \g(q))\\
 & \le M C_\g \cdot d(p, q)\\
& =M C_\g \cdot d(y_1, y_2).
\end{align*}
   Now   inequality   (\ref{e3.2}) implies  $ |h_\g(y_1)-h_\g(y_2)|^\beta\le M^2 \cdot d(y_1, y_2)$.

\end{proof}


 Let  $\Gamma^*$ be the opposite group of $\Gamma$. 
 We denote the group operation in $\Gamma^*$ by $\gamma_1 *\gamma_2$ (and the group operation in $\Gamma$ by $\gamma_1 \gamma_2$).
     We define a $\Gamma^*$-module structure 
    $\pi:  \Gamma^*\ra \text{Aut}(E)$   on $E$  as follows. For $\g\in \Gamma^*$, $h\in E$ and $y\in Y$,   define  
\begin{equation}\label{IAeqn}(\pi_\g h)(y)= a_\g^{-1} A_{\g}^{-1}  h(\sigma_\g(y))-a_\g^{-1} A_{\g}^{-1}  h(\sigma_\g(y_0)).
\end{equation}
Using   Lemma \ref{hislip}  (1)      it is not difficulty  to check that $\Gamma^*$ acts on $E$ by linear isometries.

Now define a map $b: \Gamma^*\ra E$ by
 $b\g=h_\gamma.$
   By comparing the two sides of  $(\g_1\g_2)(x,y)=\g_1(\g_2(x,y))$,
  we obtain    $h_{\g_1\g_2}=h_{\g_2}+\pi_{\g_2}h_{\g_1}$.
   In other words,      $b_{\g_2*\g_1}=b_{\g_2}+\pi_{\g_2}b_{\g_1}$.
  Hence 
  $b: \Gamma^*\ra E$ is a   cocycle.  


%

%
\begin{lemma}\label{conjcallem} 
   Let $\Gamma$   and $b$   be  as above.  
  If the cocycle $b$ is a coboundary, then  
 there exists some $h_0\in E$ such that 
if we denote by $H_0:   \R^n\times Y\ra \R^n\times Y$  the biLipschitz map   given by   
$$H_0(x,y)=( x+ h_0(y), y), $$
  then every element of  $H_0 \Gamma H_0^{-1}\subset \text{Homeo}(\R^n\times Y)$
         is a 
 similarity.
\end{lemma}
\begin{proof} 
Since $b$ is a coboundary, there is some $h_0\in E$ such that
  $h_\g=b_\g=h_0-\pi_\g h_0$ for all $\g\in \Gamma$.  
 Now for any $\g\in \Gamma$, we  have 
 \begin{align*}
H_0\circ \g\circ H_0^{-1} (x,y)&=H_0\circ \g(x-h_0(y), y)\\
&=H_0(a_\g A_\g[x-h_0(y)+x_\g+h_\g(y)], \sigma_\g(y))\\
&=(a_\g A_\g [x-h_0(y)+x_\g+h_\g(y)+a_\g^{-1}A_\g^{-1}h_0(\sigma_\g(y))],  \sigma_\g(y))\\
&=(a_\g A_\g [x-h_0(y)+x_\g+h_\g(y)+\pi_\g h_0 (y)+a_\g^{-1}A_\g^{-1}h_0(\sigma_\g(y_0))],  \sigma_\g(y))\\
&=(a_\g A_\g [x+x_\g+a_\g^{-1}A_\g^{-1}h_0(\sigma_\g(y_0))],   \sigma_\g(y)).
\end{align*}
  Notice that $x_\g+a_\g^{-1}A_\g^{-1}h_0(\sigma_\g(y_0))$ is a constant vector in $\R^n$ (independent of $(x,y)$). 
 So   
 $$x\mapsto a_\g A_\g [x+x_\g+a_\g^{-1}A_\g^{-1}h_0(\sigma_\g(y_0))]$$
     is a similarity of $\R^n$.   
  Now Lemma \ref{hislip}  (1) implies   that   $H_0\circ \g\circ H_0^{-1}$  
is a similarity of  $\R^n\times Y$.

\end{proof}

Recall that our goal is 
to show that  $\Gamma$ can be conjugated into a group of similarities.
By  Lemma \ref{l1.1}  and Lemma \ref{conjcallem},  it now suffices to show that  the  affine action 
corresponding to the cocycle $b$  has a   global fixed point.    This is where Day's fixed point  theorem   can  be useful.

\begin{theorem}\label{conjtosim}  
  Let $Y$ be a metric space,   $0<\beta\le 1$    and $\R^n \times Y$   be   equipped with the metric
$$d((x_1, y_1), (x_2, y_2))=\max\{|x_1-x_2|^\beta, d(y_1, y_2)\}\;\; \text{for}\;\; (x_1, y_1), (x_2, y_2)\in \R^n\times Y.$$
Suppose that a  locally compact group  $\Gamma$   acts 
  continuously  on $\R^n \times Y$ by  biLipschitz maps,
  and for every $\g\in \Gamma$,   there exist 
 $a_\g \in \R_{+}, A_\g \in \text{O}(n)$, $x_\g\in \R^n$, $h_\g\in E$   and a similarity   $\sigma_\g: Y\ra Y$  of $Y$  such that  
$$ \g(x,y) =( a_\g A_\g(x +x_\g+ h_\g(y)), \sigma_\g(y))\;\;\text{for}\;\; (x,  y)\in \R^n\times   Y.$$
     If $\Gamma$ is amenable   and  is a uniform quasisimilarity group  of  $\R^n \times Y$,   
  then  
 there exists some $h_0\in E$ such that 
if we denote by $H_0:   \R^n\times Y\ra \R^n\times Y$  the biLipschitz map   given by   
$$H_0(x,y)=( x+ h_0(y), y), $$
  then every element of  $H_0 \Gamma H_0^{-1}\subset \text{Homeo}(\R^n\times Y)$
         is a 
 similarity.
\end{theorem}
\begin{proof}
We equip  the vector space $E$ with the topology of pointwise convergence.  
It is easy to check that $E$ is a locally convex topological vector space. 
  Notice that   $E\subset (\R^n)^Y$ is a subset of the space of all maps from $Y$ to $\R^n$   and  
the topology of pointwise convergence on $E$ is the same as the subspace topology 
 $E\subset (\R^n)^Y$,    where $(\R^n)^Y$  has the product topology.

We observe that  closed balls in $E$ are compact  in the topology of pointwise convergence (this is similar to the fact that  closed balls in the dual space of a normed vector space is compact in the weak$^*$ topology). 
 Indeed,    for $r>0$,   the closed ball $B_r=\{h\in E: ||h||\le r\}$   is  a closed subset of 
 $$A:=\prod_{y\in Y} \bar{B}(0,  r\cdot (d(y, y_0))^{\frac{1}{\beta}})\subset (\R^n)^Y,$$ 
     where 
 $\bar{B}(0,  r\cdot  (d(y, y_0))^{\frac{1}{\beta}})$ is the closed ball in $\R^n$.  By Tychonoff's theorem, $A$ is compact.  As a closed subset of $A$,  the set   $B_r$ is also compact. 
  Consequently, all closed bounded subsets of $E$ are compact in the topology of pointwise convergence.

By Lemma \ref{hislip} (2),    there is a constant $C>0$ such that 
  $||h_\g||\le C$ for all $\g\in \Gamma^*$.    So the cocycle $b$ is bounded.
   By  Lemma \ref{l1.2},   the affine action of $\Gamma^*$ on $E$ corresponding to $b$   has bounded orbits. 
 This is also  easy to see directly since the orbit of  $0\in E$ is   $\{h_\g: \g\in \Gamma^*\}\subset B_{C}$. 
  Let $K\subset  E$ be the closed convex hull  of this orbit.  By the preceding paragraph, 
    $K$ is a compact convex subset of $E$.  Notice that $\Gamma^*$ acts on $K$ by  isometric   affine transformations.     
  Since $\Gamma$   acts 
  continuously  on $\R^n \times Y$, 
   it is easy to check that the  action 
of $\Gamma^*$  on $K$ 
    is separately continuous.    Finally $\Gamma^*$ is amenable since $\Gamma$ is. 
 Now all the conditions in Day's theorem are satisfied and so this affine action has a global fixed point.     Theorem  \ref{conjtosim}  
 follows from Lemma \ref{l1.2} (2)  and Lemma \ref{conjcallem}.

%
\end{proof}

\section{Quasiconformal groups of model Filiform groups}\label{filiform}

In this section we show through an example that   Theorem \ref{conjtosim}
  is more applicable than it appears.   The point is that the space that  $\Gamma$ acts on does not have to be a product like $\R^n\times Y$.  
  Specifically
 we show how to use Theorem \ref{conjtosim}   to  prove Theorem \ref{fili}.     

%

The $n$-step ($n\ge 2$)  model Filiform algebra $\mathfrak{f}^n$    is an $(n+1)$-dimensional  real  Lie 
 algebra.  
  It has a   basis $\{e_1, e_2, \cdots, e_{n+1}\}$ 
  and  the only non-trivial bracket relations are 
    $[e_1, e_j]=e_{j+1}$  for $2\le j\le n$.    
    The Lie algebra   $\mathfrak{f}^n$   admits a  direct sum decomposition of vector subspaces
  $\mathfrak{f}^n=V_1\oplus \cdots \oplus V_{n}$, where $V_1$ is the  linear subspace spanned by $e_1, e_2$, and $V_j$ ($2\le j\le n$)  is the linear subspace spanned by $e_{j+1}$.     It is easy to  check that 
  $[V_1, V_j]=V_{j+1}$ for $1\le j\le n$,  where $V_{n+1}=\{0\}$.   Hence $\mathfrak{f}^n$  is a 
  stratified      Lie algebra. 
 The connected and simply connected Lie group with Lie algebra 
$\mathfrak{f}^n$    will be denoted by  $F^n$  and is called the 
   $n$-step   model Filiform  group.

A group $\Gamma$ of quasiconformal maps of a metric space $X$ is called a uniform 
quasiconformal group if there is some $K\ge 1$ such that every element $\gamma$ of 
$\Gamma$ is  $K$-quasiconformal.
  By Theorem 4.7 in \cite{HK},  every $K$-quasiconformal map 
$f: N\ra N$ of a Carnot group (equipped with a left invariant Carnot metric) 
    is $\eta$-quasisymmetric, where $\eta: [0, \infty)\ra [0, \infty)
$  is a homeomorphism depending only on  $K$ and $N$.     By  Lemma 3.10  in \cite{X2},
   every $\eta$-quasisymmetric map $f: F^n\ra F^n$  for $n \ge 3$  is a $(M, C)$-quasisimilarity,
 where $M$ depends only on $\eta$.    It follows that every uniform quasiconformal group of 
 $F^n$ ($n\ge 3$)  is a uniform quasisimilarity group.




 Recall that, for a connected and simply connected nilpotent Lie group 
 $N$ with Lie algebra  $\mathfrak n$, the exponential map
 $\text{exp}: \mathfrak n\ra N$ is a diffeomorphism. We shall   identify 
 $\mathfrak f^n$ with $F^n$  via the exponential map. 
  For any $p\in \mathfrak f^n$, let $L_p: \mathfrak f^n\ra \mathfrak f^n$, $L_p(x)=p*x$   be the left translation by $p$.
    For  $a_1, a_2\in \mathbb{R}\backslash\{0\}$, let 
 $h_{a_1, a_2}:   \mathfrak{f}^n\ra \mathfrak{f}^n$ be the graded automorphism given
  by   
$$h_{a_1, a_2}(e_1)=a_1e_1,$$
 $$h_{a_1, a_2}(e_j)=a_1^{j-2} a_2 e_j  \;\;\text{for}\;\; 2\le j\le n+1.$$
 For any Lipschitz function $h:\mathbb  R\ra \mathbb  R$, define $h_j:\mathbb   R\ra \mathbb  R$ ($2\le j\le n+1$)   inductively by $h_2=h$,
 $$h_j(x)=-\int_0^x h_{j-1}(s) ds,  \;\; j=3,\cdots, n+1.$$
 Let $F_h: \mathfrak f^n\ra \mathfrak f^n$ be given by
 $$F_h(x)=x*\sum_{j=2}^{n+1}h_j(x_1)e_j,$$
  where $x=\sum_{j=1}^{n+1} x_j e_j.$
  
Let   $V_1$ be equipped with the inner product with $e_1, e_2$ as orthonormal basis,  and $\mathfrak f^n$ be equipped with 
the Carnot  metric   determined by  this inner product.  It is easy to check that   $h_{\epsilon_1,   \epsilon_2}$   with 
$\epsilon_1, \epsilon_2\in \{1, -1\}$ is an isometry of   
$\mathfrak f^n$.   
  Recall that the standard Carnot dilation 
$\delta_t:  \mathfrak f^n\ra  \mathfrak f^n$ ($t>0$)   is defined by
 $\delta_t(v)=t^j v$ for $v\in V_j$. 
 They are similarities with respect to   the Carnot metric:  
 $d(\delta_t(p), \delta_t(q))=t\cdot d(p,q)$ for any $p, q\in \mathfrak f^n$.  
 Since  $\delta_t$ ($t>0$) is a similarity and left translations are isometries  of   
$\mathfrak f^n$,   the group  $Q$ generated by   
$h_{\epsilon_1,   \epsilon_2}$   
 ($\epsilon_1, \epsilon_2\in \{1, -1\}$),  $\delta_t$ ($t>0$)  and left translations  consists of similarities of 
$\mathfrak f^n$. In fact, it is not hard to see that $Q$ is the group 
$\text{Sim}(\mathfrak f^n)$   of similarities of $\mathfrak f^n$. 
  Notice that   the identity component   $Q_0$  of  $\text{Sim}(\mathfrak f^n)$  
  consists of maps of the form $L_p\circ \delta_t$ ($p\in  \mathfrak f^n$, $t>0$),  
 and that $\text{Sim}(\mathfrak f^n)$    
has 4 connected components    $h_{\epsilon_1,   \epsilon_2}Q_0$.  

{\bf{Proof of   Theorem  \ref{fili}}}.  
    Let   $\Gamma\subset QC(\mathfrak f^n)$ be a locally compact uniform quasiconformal group.
 As   indicated   above,  $\Gamma$ is a uniform quasisimilarity group.  
By Theorem 1.1 of   \cite{X2},  every quasiconformal map $F: \mathfrak f^n\ra 
\mathfrak f^n$ has the form
 $F= h_{a_1, a_2}\circ   L_p\circ F_h$,  where  
$a_1, a_2\in \mathbb R\backslash\{0\}$,  $p\in \mathfrak f^n$  and $h: \mathbb R\ra \mathbb R$  is  a  
Lipschitz function.  
  It follows that 
every quasiconformal map $F: \mathfrak f^n\ra 
\mathfrak f^n$  induces  a biLipschitz map  $f$ of $V_1=\R e_2\times Y$  (with $Y=\R e_1$)
 of the following form 
$$f(x_2 e_2 + x_1 e_1)= a_2(x_2 +b+ h(x_1)) e_2+ a_1 (x_1+a) e_1,$$
  where $a, b\in \R$,  $a_2, a_1\in \R\backslash\{0\}$  and  $h: \R\ra \R$ is Lipschitz. 
  By replacing  $h$ with $h-h(0)$ and  $b$ with $b+h(0)$ we may assume $h(0)=0$. 
  Furthermore, it is not difficulty to show that  if $F$ is a $(M, C)$-quasisimilarity, then so is $f$. 
 Consequently,     the map  $\rho: \Gamma\ra  Homeo(\R^2)$, $F \mapsto f$,  defines an action of $\Gamma$ on $\R^2$ and 
   $\rho(\Gamma)$   is 
 a uniform quasisimilarity group.  
  Lemma \ref{hislip} (1) implies that 
  $|a_1|=|a_2|$.


\begin{claim}\label{claim}   $QC(\mathfrak f^n)$ is a  solvable  group. \end{claim}

We first finish the proof of  Theorem  \ref{fili} assuming Claim  \ref{claim}. 
It follows from Claim \ref{claim} that  $\Gamma$ is solvable and hence is amenable. 
 Now     Theorem  \ref{conjtosim}   implies that there is some Lipschitz function $h_0:\R\ra \R$ such that if we denote by $H_0: V_1\ra V_1$ the map given by 
$$H_0(x_2 e_2+x_1 e_1)= (x_2+h_0(x_1))e_2+ x_1 e_1,$$
  then all elements of  $H_0 \rho(\Gamma)  H_0^{-1}$  have the form
 \begin{equation}\label{e3.10}
x_2 e_2+x_1 e_1\mapsto    a_2(x_2 +b) e_2+ a_1 (x_1+a) e_1\;\; \text{with}\;\; |a_2|=|a_1|.
\end{equation}
Notice that $H_0$ is the map on $V_1$ induced by the 
biLipschitz map $F_{h_0}: \mathfrak f^n\ra \mathfrak f^n$.    It follows that every element of 
$F_{h_0}\Gamma F^{-1}_{h_0}$  induces a  biLipschitz map of $V_1$ as in (\ref{e3.10}).
 Notice that every  map of the form  (\ref{e3.10})   is also   induced by a biLipschitz map 
 of $\mathfrak f^n$ of the form $h_{\epsilon_1, \epsilon_2}\circ \delta_t\circ L_p$,
 where $\epsilon_1, \epsilon_2\in \{1, -1\}$,  $p\in \mathfrak f^n$ and $t>0$.   
  If    two quasiconformal maps  induce  the same  map    on $V_1$,    then they have the same Pansu differential a.e.
  By Lemma  2.5 in \cite{X2},   these two quasiconformal maps differ by a left translation.  
    Hence  every element of 
$F_{h_0}\Gamma F^{-1}_{h_0}$    has the form 
  $L_q\circ h_{\epsilon_1, \epsilon_2}\circ \delta_t\circ L_p$.    Since each of 
  $L_q$,  $h_{\epsilon_1, \epsilon_2}$,  $\delta_t$, $L_p$ is  a similarity,
   the group $F_{h_0}\Gamma F^{-1}_{h_0}$   consists of  similarities.


Next we prove the claim.  First we define a homomorphism 
 $\pi_1: QC(\mathfrak f^n)\ra  \R^*\times  \R^*$ by
  $$\pi_1(h_{a_1, a_2}\circ L_p \circ F_h)=(a_1, a_2).$$
  Since $\R^*\times  \R^*$  is  abelian, 
  $QC(\mathfrak f^n)$ is solvable if the kernel $H_1=\text{ker}(\pi_1)$  is.  
  Notice $H_1$ consists of all quasiconformal maps of $\mathfrak f^n$ of the form
 $F=L_p\circ F_h$, where $p\in \mathfrak f^n$ and  $h:\R\ra \R$ is Lipschitz  satisfying $h(0)=0$.
 Now define a homomorphism $\pi_2: H_1\ra  \R$ by
 $$\pi_2(L_{\sum_{j=1}^{n+1}x_je_j}\circ F_h)=x_1.$$
 The kernel $H_2=\text{ker} (\pi_2)$  of $\pi_2$ consists  of all maps of the form 
 $L_{\sum_{j=2}^{n+1}x_j e_j}\circ F_h$, where $x_j\in \R$ and $h: \R \ra \R$ is Lipschitz  satisfying $h(0)=0$.    A direct calculation shows  that $H_2$ is abelian.    
Since $\R$ is also abelian,   $H_1$ is solvable.   Hence  $QC(\mathfrak f^n)$ is solvable. 

\qed

Here we make some remarks on the connection  between Carnot groups (in particular, the model Filiform groups)
    and homogeneous manifolds with negative curvature.
  Recall that $\R^n$  (together with $\infty$) can be identified with the ideal boundary of the $n+1$ dimensional 
 real hyperbolic space (in the upper half space model). The Euclidean metric on $\R^n$ is a  parabolic visual metric   associated to $\infty$.  Similarly, given any Carnot group $N$ with Lie algebra $\mathfrak n=V_1\oplus \cdots \oplus V_r$, let $\R$ act on $N$ by the standard Carnot group dilations $\delta_{e^t}$ ($t\in \R$). Let $S=N\rtimes \R$ be the corresponding semi-direct product. Then $S$ is a solvable Lie group.   Fix an inner product on the tangent space at the identity element of $S$ such that $\mathfrak n$ and $\R$ are perpendicular and that  $V_i$ and $V_j$ are perpendicular for $i\not=j$.  Then $S$ is negatively curved \cite{H} and the ideal boundary of $S$ can be identified with $N\cup \{\infty\}$. Furthermore,  the Carnot metric on $N$ is a parabolic visual metric associated with $\infty$.

 \section{Applications  to quasi-isometric rigidity}\label{app:sec}

In this section we show how Theorem \ref{conjtosim} can be used to prove Tukia-type Theorem \ref{ProperTukiaThm}   and Theorem \ref{ProperTukia2} below.  Additionally we show how 
Theorem \ref{ProperTukiaThm}  
 can  be used to simplify  some of the proofs of quasi-isometric rigidity found in \cite{D1} and \cite{P1,P2}. Then we show
    how both Theorem \ref{ProperTukiaThm}   and Theorem \ref{ProperTukia2}  can be used to improve results on envelopes of abelian-by-cyclic groups found in \cite{D3}.
 Finally in Theorems \ref{s5.3.2} and \ref{lcdia}   we prove results on quasi-isometric rigidity of Lie groups and locally compact groups quasi-isometric to certain solvable Lie groups.

\subsection{Negatively curved homogeneous spaces.}\label{homogsec}


In this subsection we prove Theorem \ref{ProperTukiaThm}.

Let $A$ be an   $n\times n$ matrix with real  entries.   Suppose 
  $A$ is 
diagonalizable  over the complex numbers and   its eigenvalues have positve real parts.
 We list   the
 real parts of eigenvalues  
  in increasing order  $\alpha_1<\alpha_2< \cdots < \alpha_r$. 
  Let $\R$ act on 
$\R^n$  by the one parameter group $e^{tA}$ ($t\in \R$)  and  
$G_A= \R^n \rtimes_A \R $   the associated semi-direct product. 
  Then  $G_A $ is a   solvable Lie group.  
    Equip $G_A$  with the left invariant Riemannian metric that is determined by the standard inner product at the identity element $(0,0)\in \R^n\times \R=G_A$.  
 Then by \cite{H} we have that $G_A$ is negatively curved.   
  For $x_0\in \R^n$, the path $c_{x_0}: \R\ra G_A$, $c_{x_0}(t)=(x_0, t)$, is a geodesic in $G_A$.   We call   $c_{x_0}$ 
  a    vertical geodesic.
 All  vertical geodesics are asymptotic as $t \to \infty$, and so they determine a point $\infty$  in the ideal boundary. If $t \to -\infty$ all vertical geodesics diverge from one another. We call such geodesics downward oriented. 
  The ideal boundary  $\partial G_A$  of $G_A$ is naturally identified with 
 $\R^n\cup \{\infty\}$, where points   in $\R^n$ correspond to downward oriented vertical geodesics.

  Let 
$\bar{A}$   be the matrix  obtained from  the 
 Jordan form of $A$ by  replacing the  eigenvalues with their real parts. 
By  Proposition 4.1   of    \cite{FM}    and  Corollary 3.2 of  \cite{X1},  
the ideal boundaries of $G_A$ and $G_{\bar{A}}$ are biLipschitz. 
   It follows that 
we can assume that $A$ is already diagonal with eigenvalues exactly $\alpha_i$ listed in increasing order.  Let $V_j$  ($1\le j\le r$)  
 be  the eigenspace associated to $\alpha_i$.  Then  $\R^n=V_1\times \cdots \times V_r$.
  We write a point $x\in \R^n$ as $x=(x_1,  \cdots,   x_r)$ with $x_j\in V_j$.  
The  \emph{parabolic  visual  metric}   $d_A$  on   $\R^n=\partial G_A\backslash \{\infty\}$
    associated with  $\infty$  is given by 
$$ d_A((x_1, \ldots, x_r), (x'_1, \ldots, x'_r))= \max_i \{|x_i -x_i'|^{\alpha_1/\alpha_i}\}.$$
The notation used in \cite{D1} defined $G_A$ as $G_M$ where $M=e^A$ but we switch to writing $G_A$ instead of $G_M$ to be  in line with the notation used in \cite{X1}. 

By Proposition 4 in \cite{D1},  every  biLipschitz (and by \cite{X1} any quasi-symmetric) map of    
 $ (\R^n, d_A)$ has the form
$$ F(x_1, \ldots , x_r)=(f_1(x_1, \ldots, x_r), f_2(x_2, \ldots, x_r), \ldots, f_r(x_r)),  $$
   where $f_i$ is biLipschitz in $x_i$ and $\alpha_i/\alpha_j$-H\"{o}lder in $x_j$ for $j>i$. 

We are now able to prove   Theorem \ref{ProperTukiaThm}.

 {\bf Proof of Theorem \ref{ProperTukiaThm}.}
By Theorem 2 of \cite{D1} we have that  $\Gamma$   
 can be conjugated to act by maps that are the composition of similarities and maps of the form 
$$(x_1,\; x_2,\; \ldots, \;x_r) \mapsto ( x_1+h_1(x_2, \ldots, x_r),\; x_2+ h_2(x_3, \ldots, x_n),\; \ldots, \;x_r + h_r)$$
where $h_i$  is $\alpha_i/\alpha_j$-H\"{o}lder in $x_j$ for $j>i$. In other words,
  after conjugation by a biLipschitz  map,    any $F \in \Gamma$ has the form 
$$F(x_1, \ldots, x_r)=(a^{\alpha_1}A_1 (x_1 + h_1(x_2, \ldots, x_r)), \ldots, 
a^{\alpha_{r-1}}A_{r-1} (x_{r-1} + h_{r-1}(x_r)), 
a^{\alpha_r}A_r (x_r + h_r))$$
where  $a>0$  and  the $A_j$ is an  orthogonal  transformation of $V_j$. These maps are referred to as \emph{almost similarities} in \cite{D1}.

   Notice that     $\Gamma$ induces an action on  $V_j\times \cdots \times V_r$   for any $1\le j \le r$ and that the induced action of  $\Gamma$  on $V_r$ is   by similarities.   We will finish the proof by induction. 
  Assume   the induced action of $\Gamma$ on $V_{j+1}\times \cdots \times V_r$   is by similarities. 
 We will show that there is a biLipschitz map $F_0:   (\R^n, d_A)  \ra   (\R^n, d_A) $  such that the induced action of
  $F_0\Gamma F_0^{-1}$  on $V_j\times \cdots \times V_r$  is by similarities.  This will complete the proof.

Set  $Y=V_{j+1}\times \cdots \times V_r$.  The metric on $Y$ is given by
  $$D((x_{j+1}, \cdots, x_r),  (x'_{j+1}, \cdots, x'_r))=\max_{j+1\le k\le r} |x_k-x'_k|^{\frac{\alpha_1}{\alpha_k}}.$$
   Notice that   the map $h_j:   (Y, D)\ra (V_j, |\cdot|^{\frac{\alpha_1}{\alpha_j}})$ is Lipschitz. 
  Denote  $a_j=h_j(0)$ and let $g_j(y)=h_j(y)-a_j$.    
Then $g_j:  (Y, D)\ra (V_j, |\cdot|^{\frac{\alpha_1}{\alpha_j}})$
 is Lipschitz and $g_j(0)=0$.  
Also by assumption,  the induced action of $\Gamma$ on $Y$ is by 
 similarities.
   Now  
    the induced action of $\Gamma$ on $V_j\times Y$ is by biLipschitz maps of the form
  $$(x_j,  y)\mapsto  (a^{\alpha_j} A_j (x_j+a_j+ g_j(y)),  \sigma(y)),$$
  where $g_j:   (Y, D)\ra (V_j, |\cdot|^{\frac{\alpha_1}{\alpha_j}})$ is Lipschitz, $g_j(0)=0$ and $\sigma$ is a similarity of  $Y$.  
  Since $\Gamma$ is a uniform quasisimilarity group of  $(\R^n, d_A)$, 
    it  is easy to see that the induced action of $\Gamma$ on  $V_j\times Y$  is a
  uniform quasisimilarity action.  Hence    Theorem  \ref{conjtosim}  implies that there is some 
 Lipschitz map $h:    (Y, D)\ra (V_j, |\cdot|^{\frac{\alpha_1}{\alpha_j}})$ such that  after conjugation by the  biLipschitz map 
 $  V_{j}\times Y\ra     V_{j}\times Y$,  $(x_j,  y)\mapsto (x_j+h(y), y)$,
    $\Gamma$ acts on    $V_{j}\times Y$  by   similarities.   
  Let $F_0:  (\R^n, d_A)  \ra   (\R^n, d_A) $ be the map given by:
$$F_0(x_1, \ldots, x_j,   y) = (x_1, \ldots, x_{j} +h(y) , y).$$
 Then $F_0$ is biLipschitz. Furthermore,  
the induced action of
  $F_0\Gamma F_0^{-1}$  on $V_j\times \cdots \times V_r$  is by similarities.
\qed

\subsection{Quasi-isometric rigidity for lattices in certain solvable Lie groups}
In this  subsection we show how Theorem \ref{ProperTukiaThm} can be used to simplify parts of the proof of quasi-isometric rigidity for lattices in certain solvable Lie groups.  The first proof we can simplify is the following theorem for lattices in certain abelian-by-cyclic groups that was announced in \cite{EFW}:   

\begin{theorem}\label{specialcase} 
  Let  $A$ be a $n\times n$ real matrix 
diagonalizable over the complex numbers. Suppose    $\tr{A}=0$
    and that   $A$ has  no  purely imaginary eigenvalues. 
Let $G_A= \R^n \rtimes_A \R$. If $\Gamma$ is a finitely generated group quasi-isometric to $G_A$,  
    then $\Gamma$ is virtually a lattice in $\R^n \rtimes_{B}\R$,
        where $B$ is a matrix that has the same absolute Jordan form as $\alpha A$ for some positive $\alpha \in \R $. 
\end{theorem}

Here $A$ necessarily has both eigenvalues with positive real part and eigenvalues with
 negative  real part.
Theorem \ref{specialcase} is proved in \cite{P1,P2} as part of the following more general theorem on lattices in abelian-by-abelian solvable Lie groups.  

\begin{theorem}[Peng]\label{pengthm} Let $G_\psi= \mathbf{H} \rtimes_\psi \mathbf{A}$ where $\mathbf{H},\mathbf{A}$ are   Euclidean   groups and $\psi:\mathbf{A} \to Aut(\mathbf{H})$ is such that every (nontrivial) element of $\psi(\mathbf{A})$ is diagonalizable and has at least one eigenvalue whose absolute value is not equal to one.
Suppose that $\Gamma$ is a finitely generated group quasi-isometric to a lattice in $G_\psi= \mathbf{H} \rtimes_\psi \mathbf{A}$. Then $\Gamma$ is virtually a lattice in $\mathbf{H}\rtimes_{\psi'}\mathbf{A}$ for some $\psi':\mathbf{A} \to Aut(\mathbf{H})$.
\end{theorem}


We present the outline of the simplified proof of Theorems \ref{specialcase} and 
\ref{pengthm} now that we have access to Theorem \ref{ProperTukiaThm}.

A key part of the argument is understanding the {quasi-isometry group} of $G_\psi$.
 Let $X$ be a metric space. 
 Two quasi-isometries $f, g: X\ra X$  are  equivalent if   
 $$d(f,g):=\sup\{d(f(x), g(x))| x\in X\}<\infty.$$
  Let   $QI(X)$ be the set of equivalence classes  $[f]$ of self quasi-isometries of $X$.
  The formula $[f]\cdot [g]=[f\circ g]$ defines a group  structure on $QI(X)$.
 The inverse of $[f]$ is the class represented by a quasi-inverse of $f$. 
 We call $QI(X)$ the {\emph{quasi-isometry group} }   of $X$.  A subgroup $U$ of $QI(X)$ is called \emph{uniform} if there are fixed constants  $K\geq 1,C\geq 0$ such that each class in $U$ has at least one representative that is a $(K,C)$  quasi-isometry. 

\subsection*{Outline of proof of Theorems  \ref{specialcase} and \ref{pengthm}.}
\begin{enumerate}
\item The first ingredient is Peng's theorem on the structure of quasi-isometries $f: G_\psi \to G_\psi$ (Theorem 5.3.6 in \cite{P2}). Peng shows that all such quasi-isometries $f$ are a bounded distance from a map of the form 
$(\mathbf{x}, \mathbf{t}) \mapsto (f_{\mathbf{H}}(\mathbf{x}), f_{\mathbf{A}}(\mathbf{t}))$  for $(\mathbf{x}, \mathbf{t}) \in \mathbf{H} \rtimes_\psi \mathbf{A}$  such that $f_{\mathbf{A}} :\mathbf{A} \to \mathbf{A}$ is an affine map. 
\item Using Peng's theorem one can identify the quasi-isometry group (up to finite index) as a product of groups of biLipschitz maps:
$$ QI(G_\psi) \simeq \prod Bilip( \R^{n_i}, d_{A_i}) $$
where the $A_i$ depend on $\psi$. In the special case of Theorem \ref{specialcase} this becomes
$$QI(G_A) \simeq Bilip( \R^{n_\ell}, d_{A_\ell}) \times  Bilip( \R^{n_u}, d_{A_u}), $$
where $A_u$ corresponds to the eigenvalues of $A$  with positive real part 
 and $A_\ell$ corresponds to the eigenvalues with negative real part,  
and $n_\ell, n_u$ are the dimensions of $A_\ell$ and $A_u$ respectively. 
 (See Sections 2.4 and 2.5 in \cite{D1} for the details of the abelian-by-cyclic case and Proposition 5.3.5 (vi) in \cite{P2} for details in the general case.)
\item Any  finitely generated  group $\Gamma$ quasi-isometric to $G_\psi$ is, up to finite kernel, a uniform subgroup of $QI(G_\psi)$ and hence   acts as  a uniform quasi-similarity subgroup of each $Bilip( \R^{n_i}, d_{A_i})$.
\item   By Theorem \ref{ProperTukiaThm} we have that after conjugation
$$f\Gamma f^{-1} \subset \prod Sim(\R^{n_i}, d_{\bar A_i}).$$
Note that we are allowed to use  Theorem \ref{ProperTukiaThm} since $\Gamma$ being quasi-isometric to an amenable group (a lattice in a  solvable Lie group) is itself amenable.  
\item By uniformity, the similarity constants in each $Sim(\R^{n_i}, d_{\bar A_i})$ factor interact in such a way that $f\Gamma f^{-1}$ can be viewed as a cocompact subgroup of the isometry group of $\mathbf{H}\rtimes_{\bar{\psi}}\mathbf{A}$ for some faithful homomorphism $\bar{\psi}:\mathbf{A} \to Aut(\mathbf{H})$ whose image consists of diagonal matrices. (This is discussed in more detail in Section 4 of \cite{D1}).
\item By Proposition  \ref{witte:prop} below (due to Dave Witte Morris)  this implies that $\Gamma$ is virtually a lattice in $\mathbf{H}\rtimes_{{\psi'}}\mathbf{A}$ for some $\psi'$.
 \end{enumerate}

 {\bf Comments.} The previous proofs of Theorems \ref{specialcase} and \ref{pengthm} diverge from this outline at Step 4. Specifically, 
by Theorem 2 in \cite{D1}, after conjugation $\Gamma$ acts by \emph{almost similarities} on each of the $(\R^{n_i},d_{\bar A_i})$ factors (see the proof of Theorem \ref{ProperTukiaThm} for a definition of almost similarity).  Without Theorem \ref{ProperTukiaThm} one is forced to study the structure of almost similarities to first show that  the  group $\Gamma$ must be virtually polycyclic (and hence virtually a lattice is some solvable Lie group). Further analysis is then used to show that it must be a lattice in $\mathbf{H}\rtimes_{\psi'}\mathbf{A}$ for some $\psi':\mathbf{A} \to Aut(\mathbf{H})$. In the special case (i.e. Theorem \ref{specialcase}),   this is done in Section 4 of \cite{D1}. For the general case this is discussed in Corollaries 5.3.9  and 5.3.11 in \cite{P2}. In both cases it requires substantial additional analysis.

 The following proposition is due to Dave Witte Morris. 
 
 \begin{prop}\label{witte:prop}
 Fix a left-invariant metric on the semidirect product $G = \mathbf{H}\rtimes_{\bar{\psi}}\mathbf{A}$, where $\bar{\psi}$ is a faithful  homomorphism from $\mathbf{A}$ to the diagonal matrices in $Aut(\mathbf{H})$.
 Then any lattice~$\Gamma$ in $\Isom(G)$ has a finite-index subgroup that is isomorphic to a lattice in some semidirect product $ \mathbf{H}\rtimes_{{\psi'}}\mathbf{A}$.
\end{prop}

\begin{proof}
For convenience, let $I = \Isom(G)$. From \cite[Cor. 1.12 and Thms. 4.2 and 4.3]{GW}, we have $I = G \rtimes K$, where $K$ is a compact subgroup of $\Aut G$. Since $K$ is compact, it has only finitely many components, so, by passing to a finite-index subgroup, we may assume $\Gamma$ is contained in the identity component $I^\circ$ of~$I$. 

Since $\bar{\psi}$ is faithful, it is easy to see that $\mathbf{H} = \nil G$, so $(\Aut G)^\circ$ centralizes $G/\mathbf{H}$ (and we have $\mathbf{H} \subseteq \nil I^\circ$). Combining this with the observation that $K$ (being compact) acts reductively, and, being a group of automorphisms, acts faithfully, we conclude that $K$ acts faithfully on $\mathbf{H}$. This implies that $\mathbf{H} = \nil I^\circ$. Hence, a theorem of Mostow \cite[Lem. 3.9]{M} tells us that $\Gamma \cap \mathbf{H}$ is a lattice in $\mathbf{H}$. 

We may assume, after replacing $\mathbf{A}$ by a conjugate, that $K^\circ$ centralizes $\mathbf{A}$ (cf.\ \cite[Prop. 11.23(ii), p. 158]{B}). Then $I^\circ = \mathbf{H} \rtimes (\mathbf{A} \times K^\circ)$, so there is a natural projection $\pi \colon I^\circ \to \mathbf{A} \times K^\circ$. Since the kernel of $\pi$ is $\mathbf{H}$, the conclusion of the preceding paragraph implies that $\pi(\Gamma)$ is a lattice in $\mathbf{A} \times K^\circ$.

After replacing $\Gamma$ by an appropriate finite-index subgroup, it is not difficult 
to see that there is a closed subgroup $\mathbf{T}$ of $\mathbf{A} \times K^\circ$, such that $\mathbf{T} \iso  \mathbf{A}$  and $\pi(\Gamma)$ is a lattice in~$\mathbf{T}$. 
Then $\Gamma$ is a lattice in $\mathbf{H} \rtimes \mathbf{T} \simeq \mathbf{H} \rtimes_{\psi'} \mathbf{A}$. This is the desired conclusion.
\end{proof}

\subsection{Boundaries of amenable hyperbolic locally compact groups}

In this  subsection we prove a Tukia-type
     theorem  (Theorem \ref{ProperTukia2})
for boundaries of millefeuille spaces. 

  We  
consider   locally compact    uniform quasisimilarity groups of 
the metric space $\R^n \times \mathbb{Q}_m$ with the metric $d_{A,m}=\max\{d_A , d_{\mathbb{Q}_m}\}$ where $d_A$ is the metric on $\R^n$ defined in Section \ref{homogsec}  and $d_{\mathbb{Q}_m}$ is the usual metric on the $m$-adics $\mathbb{Q}_m$:
$$d_{\mathbb{Q}_m}(\sum a_i m^i, \sum b_i m^i)= m^{-(k+1)},  $$
where $k$ is the smallest index for which $a_i \neq b_i.$ As promised in the introduction we will explain how $(\R^n\times \Q_m, d_{A,m})$ is the boundary of a  \emph{millefeuille space} $X_{\varphi,m}$. A millefeuille space is a fibered product of a negatively curved homogeneous space $G_\varphi=N \rtimes_\varphi \R$ and an $m+1$ valent tree $T_{m+1}$ with respect to height functions on both factors. On $G_\varphi$ the height function $h_\varphi$ is given by projecting to the $\R$ coordinate and on the tree $h_m$ is given by fixing a base point in the ideal boundary and orienting all the edges   towards this base point. 
    Then $X_{\varphi, m}= \{ (g,t) \in G_{\varphi} \times T_{m+1}  \mid h_\varphi(g)=h_m(t)\}$. Alternatively one can construct $X_{m,\varphi}$ inductively by identifying $m$ copies of $G_\varphi$ above integral heights. Then $X_{\varphi, m}$ is a $CAT(-1)$ space with boundary $(N \times \Q_{m})\cup \{\infty\}$ and it is easy to see that $d_{\varphi,m}$ is the resulting parabolic    visual  metric. For more details see \cite{D2} and \cite{C}.

%


The following  theorem strengthens part of Theorem 1.6 in \cite{D2}.

\begin{theorem}\label{ProperTukia2} 
 Let $A$ be a  real  $n\times n$ matrix  diagonalizable over complex numbers and   whose eigenvalues have positive real   part. 
  Let  $\Gamma$ be a   separable  locally compact  uniform quasisimilarity group of  $(\R^n \times \mathbb{Q}_m, d_{A,m})$.  Suppose  
$\Gamma$ is  amenable  
  and  acts cocompactly   on the space of distinct pairs. Then, 
there exist some $\lambda>0$ and   some integer $s\ge 1$  
and  a biLipschitz map
   $F_0:    (\R^n \times \mathbb{Q}_m, d_{A,m})     \ra  (\R^n \times \mathbb{Q}_s, d_{\lambda \bar{A},s})$,
such that $m,s$ are powers of a common  integer,   and  
  $\Gamma':=F_0\Gamma F_0^{-1}$
 acts on $\R^n \times \mathbb{Q}_s$ by similarities.
   Here $\bar{A}$    denotes  the matrix  obtained from  the 
 Jordan form of $A$ by  replacing the  eigenvalues with their real parts. 
\end{theorem}
\begin{proof}
As with the proof of Theorem \ref{ProperTukiaThm} we start with existing partial results. By Theorem 1.6 in \cite{D2} 
 there exist some $\lambda>0$ and   some integer $s\ge 1$, 
and  a biLipschitz map
   $$F_0:    (\R^n \times \mathbb{Q}_m, \;d_{A,m})     \ra  (\R^n \times \mathbb{Q}_s,\; d_{\lambda \bar A,s})$$
  such that $m,s$ are powers of a common integer,    and  
  $\Gamma':=F_0\Gamma F_0^{-1}$
 acts on $\R^n \times \mathbb{Q}_s$ by maps that are similarities composed with maps of the form 
$$ (x_1, \;\ldots,\; x_r, \;y)\mapsto ( x_1+ h_1(x_2 \ldots, x_r,y), \; \ldots, \; x_r+h_r(y),  \; \sigma(y)).$$
where  $\sigma$ is an isometry of $\mathbb{Q}_s$, 
  and   
 $$h_j:   (V_{j+1}\times \cdots \times V_r\times Q_m, D_{j+1})  \ra  (V_j, |\cdot|^{\frac{\alpha_1}{\alpha_j}})$$
   is   Lipschitz,
  where the metric   $D_{j+1}$  is given by
 $$D_{j+1}((x_{j+1}, \cdots, x_r,  u),   (x'_{j+1}, \cdots, x'_r,  u'))=\max\{d_{\mathbb{Q}_s}(u, u'),  |x_k-x'_k|^{\frac{\alpha_1}{\alpha_k}}, j+1\le k\le r\}.$$
In other words any $F \in  \Gamma'$  
 has the form 
$$F(x_1, \ldots, x_r,y)=(a_F^{\alpha_1}A_{1,F} (x_1 + h^F_1(x_2, \ldots, x_r,y)), \ldots, 
, a_F^{\alpha_r}A_{r,F} (x_r + h^F_r(y)),   
\sigma_F(y))$$
where   $a_F>0$,  the $A_{i,F}$ are orthogonal matrices of the appropriate size, and  $\sigma_F$ 
 is a 
 similarity. 

We proceed by induction as in the proof of Theorem \ref{ProperTukiaThm}. 
  The action of $\Gamma'$ on $Q_s$ is already by similarities.     Now assume 
the induced action of $\Gamma'$ on $Y:=(V_{j+1}\times \cdots \times V_r\times Q_s, D_{j+1})$ is by similarities.  
 We shall find a biLipschitz map $H_0$  of   $(\R^n\times Q_s,  d_{\lambda \bar A,s})$   such that   the induced action of 
  $H_0\Gamma' H_0^{-1}$   on $(V_j, |\cdot|^{\frac{\alpha_1}{\alpha_j}})  \times Y$ is by similarities.  This will complete the proof. 
 For this, we note that the induced action of $\Gamma'$ on 
  $(V_j, |\cdot|^{\frac{\alpha_1}{\alpha_j}})  \times Y$
 is a uniform quasisimilarity action  and has the form required by 
Theorem  \ref{conjtosim}.  Since $\Gamma'$ is amenable,  Theorem  \ref{conjtosim}  implies  there is some Lipschitz map
   $h:  Y\ra (V_j, |\cdot|^{\frac{\alpha_1}{\alpha_j}}) $ such that    
  if we define $H_0:  (\R^n \times \mathbb{Q}_s, d_{\lambda \bar A,s})\ra  (\R^n \times \mathbb{Q}_s, d_{\lambda \bar A,s})$
 by 
 $$H_0(x_1,  \cdots, x_j, y)=(x_1, \cdots, x_j+h(y), y)\;\;\text{for}\;\; x_i\in V_i, \;y\in Y,$$
 then  the induced action of  $H_0\Gamma'H_0^{-1}$  on  $(V_j, |\cdot|^{\frac{\alpha_1}{\alpha_j}})  \times Y$ is by similarities.

\end{proof}

{\bf Remark.} In the above theorem we are sometimes forced to take $s$ different from $m$:
    as observed in \cite{MSW} there are quasisimilarity groups of $\Q_m$ that cannot be conjugated into the group of similarities of $\Q_m$ but  can always be conjugated  into the similarity group of some $\Q_s$,
    where $s,m$ are powers of a common  integer. Furthermore, by Corollary 5 in \cite{D2} we know that we must have $s=m^\lambda$.




\subsection{Envelopes}
In this  subsection we show how Theorems \ref{ProperTukiaThm} and \ref{ProperTukia2} can be used  to answer the following problem for certain  finitely generated solvable groups. 

\begin{problem}\label{lccmpt}
Given a finitely generated group $\Gamma$, classify up to extensions of and by compact groups, all locally compact groups $H$ such that $\Gamma \subset H$ as a  lattice. 
\end{problem}

Such an $H$ was called an \emph{envelope} of $\Gamma$  by Furstenberg in \cite{Furs} where he proposed the study of this problem. 
In \cite{F}, Furman classifies all locally compact envelopes of lattices in semisimple Lie groups and outlines a technique using the  {quasi-isometry group} to solve Problem \ref{lccmpt} for cocompact lattice embeddings.  In \cite{D3} the first author adapts this outline for cocompact lattices to solve the envelopes problem for various classes of solvable groups. Below we  show how Theorems \ref{ProperTukiaThm} and \ref{ProperTukia2} simplify the proofs from \cite{D3} in certain cases and extend   the results  to other groups as well.

The outline for approaching this problem is as follows. If $\Gamma \subset H$ is a {cocompact} lattice embedding then one can construct a map $\Psi: H \to QI(\Gamma)$ where the image of $H$ sits as a uniform subgroup
 inside $QI(\Gamma)$ and where $\Psi(\Gamma)$ sits inside $\Psi(H)$
  as a subgroup of isometries. 
At this point we should remark that this construction is only useful when the kernel of the map $H \to QI(\Gamma)$ is compact. This is not true in general but it does hold for so called \emph{QI-tame} spaces. For more details on the construction of $\Psi$ see \cite{F} Section 3 and for the definition of QI-tame see \cite{D3} Section 4.
If $QI(\Gamma)$ can be identified with a  group of quasiconformal or biLipschitz maps of some sort of boundary then Tukia-type theorems can be used
  to   conjugate the image of $H$ into a subgroup that can be identified with the isometry group of some proper geodesic metric space. The conclusion is then that $H$ must be a cocompact subgroup of  the isometry group of this metric space up to possibly some compact kernel. In \cite{F} Furman applies this outline to uniform lattices in the isometry groups of  real and complex hyperbolic spaces.

This outline works in more generality. For any locally compact compactly generated group $H$ (with or without lattices) quasi-isometric to a metric space $X$, one can similarly define $\Psi:H \to QI(X)$ so that $\Psi(H)$ is a uniform subgroup of $QI(X)$. If $H$ (and hence $X$) are QI-tame then there is a natural topology on uniform subgroups of $QI(X)$ and the map $\Psi:H \to \Psi(H)\subset QI(X)$  is a continuous map with respect to this topology. Again, if it is possible to identify $QI(X)$ with biLipschitz or quasiconformal maps of appropriate boundaries then Tukia-type theorems can be used to show that $\Psi(H)$ is a subgroup of a similarity group (which can be further identified with the isometry group of $X$ or potentially a related metric space).  This fact is used in the proof of Theorems \ref{s5.3.2} and \ref{lcdia} in the next section.


\subsubsection{Envelopes of abelian-by-cyclic groups} \label{s4.1.1}
In our case the above 
   outline in conjunction with Theorems \ref{ProperTukiaThm}  and \ref{ProperTukia2}  is used to answer Problem \ref{lccmpt} for various finitely presented abelian-by-cyclic groups. Finitely presented abelian-by-cyclic groups can be 
specified by picking $M=(m_{ij}) \in GL_n(\Z)$ and setting
$$\Gamma_M=\left< a, b_1, \ldots, b_n \mid ab_ia^{-1}= b_1^{m_{1i}}\cdots b_n^{m_{ni}},\ b_ib_j=b_jb_i\right>.$$ 
When $M \in SL_2(\Z)$ has eigenvalues off of the unit circle then $\Gamma_M$ is a lattice in $SOL$ and more generally  when $M \in SL_n(\Z)$ these groups are (virtually) lattices in {abelian-by-cyclic} solvable Lie groups. 
Alternatively, when $M$ is a one-by-one matrix with entry $m$ then $\Gamma_M$ is the solvable Baumslag-Solitar group $BS(1,m)$ but in general their geometry is much more rich.
In particular they have model spaces that are fibered products of a solvable Lie group and a tree $T_{d+1}$ of valence $d+1$ where 
$d=|\det{M}|$ with respect to appropriate {height} functions. (This time the height function on the tree is the negative of the height function we used to construct the millefeuille space in the previous section which results in a space that is not $CAT(-1)$.) 
The solvable Lie group in the fibered product can be chosen as follows:
up to possibly squaring $M$ (which is equivalent to replacing $\Gamma_M$ with an index $2$ subgroup) we can assume that $M=e^A$ lies on a one parameter subgroup,    and 
then the solvable Lie group can be chosen to be
 $G_{{A}}=\R^n \rtimes_{{A}} \R$ where $\R$ acts on $\R^n$ by the one parameter subgroup $e^{t{A}}$. (Here $A$ may have both  eigenvalues with  positive real parts and 
eigenvalues with
    negative real parts so  $G_{{A}}$ is not necessarily negatively curved). 
  Let  $X_{{M}}$  be the resulting  fibered product  of  $G_A$ and $ T_{d+1}$, 
   where $d=|\det{M}|$.  Then $X_{{M}}$ is a 
   model space for  $\Gamma_M$.  
   Similarly    we let $X_{{\bar M}}$  be the   fibered product  of  $G_{\bar A}$ and $ T_{d+1}$, 
   where  $\bar M=e^{\bar A}$  and $d=|\det{M}|$. 
%
Note that when $\det{M}=1$ then the tree in the fibered product is simply a line and so $X_{{M}} \simeq G_{{A}}$.
For more details on this see Section 6 in \cite{D3}.
 In \cite{D3} we prove the following theorem.

\begin{theorem}\label{abcthm} Suppose  $M$ has all eigenvalues off of the unit circle and either $\det{M}=1$ or all of the eigenvalues have norm greater than one.  
 Then any envelope   $H$  of $\Gamma_M$  is, 
 up to compact groups,   a cocompact  closed subgroup $H'$ in  $Isom(X_{\bar{M}^k})$ for some $k\in \Q$.
\end{theorem}
In this theorem  by ``up to compact groups'' we mean that there is a short exact sequence:
$$1 \to H' \to H/K\to K' \to 1$$ with $K,K'$ compact.   Furthermore, we may choose $H'$ such that $\Gamma_M$ embeds in $H'$.
If instead we assume that $M$ (and hence $A$) is diagonalizable then we can use Theorems \ref{ProperTukiaThm} and \ref{ProperTukia2} to simplify the proof and strengthen the statement of Theorem \ref{abcthm} and to extend it to more cases. 

\begin{theorem}\label{abcthm2} Suppose $M$ is diagonalizable over the complex numbers  and with all eigenvalues off of the unit circle.
 Let  $H$ be  an envelope of $\Gamma_M$.
  Then there is a compact normal subgroup $N$ of $H$ such that 
 $H/N$ is 
isomorphic  to a   cocompact subgroup of  
 $Isom(X_{\bar{M}^k})$ for some $k\in \Q$. 
\end{theorem}

This theorem is stronger than the previous one in that now we can conclude that $H/N$ is isomorphic to a cocompact subgroup of $Isom(X_{\bar{M}^k})$. 

\subsection*{Outline of proof of Theorem \ref{abcthm2}. }
Since most of the set up for this theorem is in \cite{D3} we will only sketch the argument and refer the reader to \cite{D3} for details.    Let $H$ be an envelope of $\Gamma_M$.
\begin{enumerate}
\item Following the outline given in the preamble above we can view $H$ (up to compact kernel) as a subgroup of $QI(\Gamma_M)$  via the map $\Psi: H \to QI(\Gamma_M)$.
\item From Section 7 in \cite{D3} we have the identification $$QI(\Gamma_M)  \simeq Bilip(\R^{n_\ell},d_{A_\ell})\times Bilip(\R^{n_u}\times \Q_{\det{M}} , d_{A_u, \det{M}}).$$
\item Since $H$ is a uniform subgroup of $QI(\Gamma_M)$ it projects to uniform quasi-similarity subgroups of 
$Bilip(\R^{n_\ell},d_{A_\ell})$
 and 
$Bilip(\R^{n_u}\times \Q_{\det{M}} , d_{A_u, \det{M}})$.
\item By Theorems \ref{ProperTukiaThm}  and \ref{ProperTukia2} we can conjugate these actions to similarity actions 
$$ fHf^{-1} \subset Sim(\R^{n_\ell},d_{\bar{A}_\ell})\times Sim(\R^{n_u}\times \Q_{(\det{M})^k} , d_{k\bar{A}_u, (\det{M})^k}).$$
\item
 Since $fHf^{-1}$ must still be a uniform subgroup of $QI(\Gamma_M)$ we see that it must actually sit inside $\Isom(X_{\bar{M}^k})$.
\end{enumerate}
{\bf Comments.}  Without Theorems \ref{ProperTukiaThm}  and \ref{ProperTukia2} the proof in \cite{D3} diverges from this proof after Step 3. Instead the partial Tukia-type results of Theorem 2 in \cite{D1} and Theorem 1.6 in \cite{D2} are used along with substantial additional analysis to find the conjugating map $f$.

\subsubsection{Envelopes of Filiform-by-cyclic groups.}
 Let $N$ be a connected and simply connected nilpotent Lie group with Lie algebra 
 $\mathfrak n$.  We call $N$ and $\mathfrak n$ {\it{rational}} if $\mathfrak n$ has a basis   with rational  structure constants. 
 It is well known that $N$ admits lattices if and only if it is rational. 
 In particular,  if $e_1, \cdots, e_n$ is a basis of $\mathfrak n$ with rational structure constants,  then there exists a positive integer $K$ such that
 the integral linear combinations of $K e_1, \cdots, K e_n$ is a lattice in $N$ (after identification of $\mathfrak n$ and $N$ via the exponential map), see 
  \cite{CG}, Theorem 5.1.8.

 
The model  Filiform group $\mathfrak f^n$ is rational.  Let $e_1, \cdots, e_{n+1}$ be the standard basis of $\mathfrak{f}_n$. As indicated above, there exists a positive integer $K$ such that 
 the set $L$ of  integral linear combinations of $K e_1, \cdots, K e_{n+1}$ 
 is a lattice in $\mathfrak f^n$.   
   Let $M=\mathfrak f_n/L$ be the quotient.  Then $\pi_1(M)=L$.

 If $\lambda\ge 1$ is an  integer, then the standard Carnot group 
 dilation $\delta_\lambda: \mathfrak f^n\ra \mathfrak f^n$ maps $L$ into $L$  and $\phi:=\delta_\lambda|_L: L\ra L$ is an injective homomorphism.  Let $\Gamma=L_\phi$ be the associated HNN extension.   
  Clearly  $\delta_\lambda$   projects to a covering map $f: M\ra M$. Furthermore, 
 $f_*: \pi_1(M)\ra \pi_1(M)$ agrees with $\phi$ after the identification $\pi_1(M)=L$. 
Let $M_f$ be the  mapping torus  of $f$. Then 
 $\pi_1(M_f)=\Gamma$ and $M_f$ is a  $K(\pi, 1)$ for $\Gamma$.

For any integer $m\ge 1$, let $X_{\mathfrak f^n, m}=\mathfrak f^n \times T_{m+1}$.  
 The metric  on   $X_{\mathfrak f^n, m}$ is defined as follows.
 Fix an end $\xi_o$ of the tree $T_{m+1}$ and let $b: T_{m+1}\ra \R$ be a Busemann function associated to $\xi_0$.  A complete geodesic in $T_{m+1}$ is vertical if one of its ends is $\xi_0$. We identify vertical geodesics   with $\R$ using the Busemann function $b$.  In this way, for each vertical geodesic $c$, we identify 
 $\mathfrak f^n\times c$ with the negatively curved solvable Lie group $S=\mathfrak f^n\rtimes \R$ defined at the end of Section \ref{filiform}.
    We equip $X_{\mathfrak f^n, m}$  with the path metric. 
   The space $X_{\mathfrak f^n, m}$  is similar to the space $X_M$ from subsection \ref{s4.1.1} in the case when all eigenvalues of $M$ lie outside the unit ball.
  When $m=[L: \phi(L)]$,  where $L$ is the lattice in $\mathfrak f^n$ described above,
  $X_{\mathfrak f^n, m}$  is the universal cover of $M_f$.

Now the argument in the proof of  Theorem \ref{abcthm2}   (using Theorem \ref{fili} instead of Theorem \ref{ProperTukiaThm})     shows the following:


\begin{theorem}
  Let  $L$ be  a lattice in $\mathfrak f^n$ constructed above  and $\Gamma$ an HNN extension as above.  Suppose $n\ge 3$.  Then for every envelope $H$ of $\Gamma$, there is a compact normal subgroup $N$ of $H$ such that $H/N$ is isomorphic to a cocompact
 subgroup of ${\text{Isom}}(X_{\mathfrak f^n, s})$  for some integer $s\ge 2$. 

\end{theorem}

\subsection{Lie groups quasi-isometric to certain solvable Lie groups}

 Tukia-type theorems can be used to characterize Lie groups and locally compact groups quasi-isometric to a given solvable Lie group. In this  subsection  we illustrate this using 
  Theorem \ref{fili} and 
 Theorem \ref{ProperTukiaThm}.

Notice that   any  two  left invariant Riemannian metrics on a connected Lie group are biLipschitz equivalent. 
Recall that we identify the model filiform algebra $\mathfrak f^n$ and the model filiform group $F^n$ via the exponential map. 
  Also recall that a locally compact group $G$  is compactly generated if there is a compact  neighborhood  $K$ of the identity element that generates $G$.   Similar to the word metric on a finitely generated group,   
   a word  pseudometric on $G$   can be defined using 
the compact generating set $K$,  and different generating sets result in quasi-isometric word  pseudometrics.

\begin{theorem}\label{s5.3.2}
 Let $n\ge 3$  and $S=\mathfrak f^n\rtimes \R$ be the semidirect product  associated to the standard action of $\R$ on the  model filiform
   group 
 $\mathfrak f^n$  by dilations.  \newline
 (1)   Let $G$ be a 
  connected and simply connected solvable Lie group. 
   Let $S$ and $G$ be equipped with left invariant   Riemannian metrics. If $G$ and $S$ are quasi-isometric, then they are isomorphic;\newline
  (2)   If a compactly generated locally compact group $G$ is  quasi-isometric to   $S$, then there is a  compact 
 normal subgroup $N$ of $G$,   such that   $G/N$  is  isomorphic to   a cocompact subgroup of 
$\text{Isom}(S)$.   

\end{theorem}

\begin{proof}
   (1)  Let $f: G\ra S$ be a    quasi-isometry and $\tilde f: S\ra G$ a quasi-inverse  of $f$.   For any $x\in G$,   let 
$L_x: G\ra G$ be the left translation by $x$.    Since $L_x$ is an isometry, there exist $L\ge 1$ and $A\ge 0$ such that 
  for every $x\in G$,    the  map 
$$T_x:=f\circ L_x\circ \tilde f: S \ra S$$ is a   $(L, A)$-quasi-isometry.
   So   $T_x$   induces a boundary map $\phi(x):   \partial S \ra \partial S$, which is  a quasiconformal map.  Here $\partial S$ is equipped with a visual metric. 
  In this way we get  a  continuous  group homomorphism $\phi:  G\ra QC(\partial S)$
  and the image  $\phi(G)$ is a  uniform quasiconformal group of $\partial S$. 
   The kernel $K:=\ker(\phi)$ is a   closed normal subgroup of $G$.

 For any $s\in S$, consider three geodesic  rays starting at $s$  such that the angle between any two of them is   at least $\pi/2$.
For each $k\in K$,  $T_k$ is a  $(L, A)$-quasi-isometry   and   induces the identity map on the ideal boundary.
 By using stability of quasigeodesics,
       it is easy to see that there is a constant $C$ depending only on  $L, A$ and the Gromov hyperbolicity constant of $S$ such that  
$d(s, T_k(s))\le C$ for all  $s\in S$.
  Since $f$ is a quasi-isometry,  it follows that $K$ is bounded.  Hence $K$ is a  compact subgroup of $G$.
Since the only compact subgroup of a connected and simply connected solvable Lie group is the 
  trivial subgroup, we have $K=\{e\}$.  
   So $\phi$ is an embedding.

  Recall that  $\partial S$ can be identified with $\mathfrak f^n\cup \{\infty\}$.
   By the results in \cite{X2},  each $F\in QC(\partial S)$  fixes $\infty$ and $F|_{\mathfrak f^n}: \mathfrak f^n \ra \mathfrak f^n$ is a quasiconformal map  with the same quasiconformality constant.
 So we may assume $G$ is a uniform quasiconformal group of $\mathfrak f^n$.   
 By   Theorem \ref{fili},  we may assume   that   $G$ is a subgroup of the similarity group  $\text{Sim}(\mathfrak f^n)$  of $\mathfrak f^n$.
   Recall  that $\text{Sim}(\mathfrak f^n)$ has 4 connected components and the identity component 
 $Q_0$  of  $\text{Sim}(\mathfrak f^n)$  
  consists of maps of the form $L_p\circ \delta_t$    ($p\in  \mathfrak f^n$, $t>0$).    Notice that 
  $S=\mathfrak f^n\rtimes \R$ 
is isomorphic to $Q_0$ via the isomorphism    $(p,t)\mapsto L_p\circ \delta_{e^t}$.
 Since $G$ is connected, we may assume $G\subset Q_0\cong S$.  However, it is easy to see that the only connected subgroup of $S$ that is quasi-isometric to $S$ is $S$ itself.

(2)     follows from the above argument.
\end{proof}


It is easy to identify $\text{Isom}(S)$.  Let $S=\mathfrak f^n\rtimes \R$ be equipped with the left invariant  Riemannian   metric such that  $\mathfrak f^n$ and $\R$ are  perpendicular to each other and such that  $e_1, \cdots, e_{n+1}$ are orthonormal.    Then 
  $\text{Isom}(S)$ is isomorphic to $\text{Sim}(\mathfrak f^n)$, which is described in Section \ref{filiform}    before the proof of Theorem \ref{fili}.  The reason is that each isometry of $S$ induces a similarity of $\mathfrak f^n$ and each similarity of $\mathfrak f^n$  is the boundary map of an isometry of $S$.

 The argument in the proof of Theorem \ref{s5.3.2}
also shows  the  following:

\begin{theorem}\label{lcdia}
    Let $A$ and $\bar A$ be as in Theorem \ref{ProperTukiaThm}.\newline
 (1)  If  a connected and simply connected solvable Lie group is quasi-isometric to 
 $G_A$, then it is isomorphic  to   a cocompact subgroup of   
$\text{Isom}(G_{\bar A})$; \newline
(2)  If a compactly generated locally compact group $G$ is  quasi-isometric to   $G_A$, then there is a  compact 
 normal subgroup $N$ of $G$,   such that   $G/N$  is  isomorphic to   a cocompact subgroup of 
$\text{Isom}(G_{\bar A})$.

\end{theorem}






 \addcontentsline{toc}{subsection}{References}

\end{document}